\documentclass{amsart}
\usepackage{amsmath,amsthm,amsfonts,amssymb,amscd}
\usepackage{graphics}
\usepackage{graphicx}
\usepackage{tikz}
\usepackage{enumerate}
\usetikzlibrary{fadings}
\usetikzlibrary{shapes.geometric}

\newtheorem{thm}{Theorem}[section]
\newtheorem{lem}[thm]{Lemma}

\newtheorem{defin}[thm]{Definition}
\newtheorem{rem}[thm]{Remark}
\newtheorem{prop}[thm]{Proposition}

\def\var{\ensuremath{\text{var}}}
\def\osc{\ensuremath{\text{osc}}}

\DeclareMathOperator*{\esssup}{ess\,sup}
\DeclareMathOperator*{\essinf}{ess\,inf}

\begin{document}
\title[Singular hyperbolic attractors are statistical stable]{Singular hyperbolic attractors are statistical stable}

\author[M. Fanaee]{Mohammad Fanaee}
\address{Instituto de Matem\'atica e Estat\'istica, Universidade Federal Fluminense, Brasil }
\email{mfanaee@gmail.com}

\author[M. Soufi]{Mohammad Soufi}
\address{Instituto de Matem\'atica e Estat\'istica, Universidade do Estado do Rio de Janeiro, Brasil }
\email{msoufin@gmail.com}

\date{}

\begin{abstract}
We prove that the unique SRB measure for a singular hyperbolic attractor depends continuously on the dynamics in the weak$^\ast$ topology.   
\end{abstract}

\subjclass[2010]{Primary 37C10; Secondary 37C40, 37D45}

\keywords{Lorenz attractor, Lorenz map, Poincar\'e section, SRB measure, Statistical stability}

\maketitle

\section{Introduction}
A common agreement on the definition of chaos is the \emph{sensitive dependence on initial
conditions}. That means independent of how close two initial conditions are, by letting
the system to proceed for a while, the new resulting states of the system are significantly
different. In other words, a small error at a starting point will cause a huge difference in the
outcome of the system. Since measuring a starting point can not be done accurately, the
orbit of states is quite unpredictable. But statistically there is a hope to make a prediction
by measuring an observable along orbits of the system. Despite of the alteration of the
observable along an orbit, its time average for typical points converges to a constant which
is the space average. This is due to the existence of an \emph{SRB (Physical) measure}. Now,
an interesting question is if the space average depends sensitively on system, i.e., if the
statistical behaviour is stable under the small perturbation of a system?\\ 

The chaos may happen when a system exhibits \emph{hyperbolicity}. The hyperbolicity can be characterize by existence of the expanding and contracting directions for the derivative map. This provides some local information about the dynamics, in a way that, in a small neighborhood of each point there are the local stable and local unstable sets, such that points on them by iteration get close or move away form each other with a uniform rate respectively. A weaker notion of hyperbolicity called \emph{Singular hyperbolicity} has been introduced in \cite{morales-pacifico-pujals04}. It is founded on C$^1$-robust transitive sets with a finite number of singularities of a flow on a compact three-manifold. The \emph{Lorenz attractor} is an example of a singular hyperbolic attractor. In \cite{lorenz63} Lorenz studied a simple model of a system of differential equations in $\mathbb{R}^3$ and showed numerically that its solution has sensitive dependence on the initial conditions near an attractor. To understand the structure of the Lorenz attractor, Williams and Guckenheimer \cite{guckenheimer-williams79} abstracted some geometric properties from the structure of the Lorenz attractor and defined the \emph{geometric Lorenz attractor}: the first example of robust transitive attractor containing a hyperbolic singularity. The singularity is accumulated by regular orbits which prevent the attractor to be (uniform) hyperbolic.

In \cite{araujo-pacifico-pujals-viana09}, it is shown that singular hyperbolic attractors are chaotic, i.e., they are sensitive dependent on initial conditions, and support an unique SRB measure. The strategy is to study the continuous dynamic through a discrete dynamic using a family of Poincar\'e sections foliated by stable manifolds in the trapping neighborhood of the attractor. Then the dynamic of the Poincar\'e map can be studied through a one-dimensional piecewise expanding map obtained by collapsing to a point the stable leaves in each Poincar\'e section. The dynamic of a piecewise expanding one-dimensional map is well understood. Under the condition of the bounded variation \cite{hofbauer-keller82} or H\"older continuity \cite{keller85} of inverse of its derivative, the map has an absolutely continuous invariant measure. This measure is used to construct an invariant SRB measure for Poincar\'e map that can be saturated along the flow in order to have an invariant SRB measure supported on the attractor. Even though it not possible to predict the long-term behavior of a trajectory, the existence of the SRB measure implies that the time average of any observable along the trajectory of almost every point converges to its space average (statistical prediction). In this work, we show that a ``small perturbation" of a vector field supporting a singular hyperbolic attractor may not cause a ``big change" in the space average of an observable (statistical stability).\\\\
\textbf{Main Theorem}. Singular hyperbolicity is statistical stable: the unique SRB measure of any singular hyperbolic attractor varies continuously in weak$^\ast$ topology.\\\\
Note that to prove the Main Theorem, it suffices to prove the same statement for SRB measure of deduced one dimensional map mentioned above. indeed, the rest follows the proofs of \cite{alves-soufi14} lifting this result to Poincar\'e section and then to flow.

\section{Singular hyperbolic attractor}
For more details on this chapter see \cite{araujo-pacifico-pujals-viana09}, \cite{keller85}, and \cite{arroyo-pujals07}  .
Let $X$ be a $C^1$ vector field on a closed three dimensional manifold $M$. 

\begin{defin}
A compact invariant set $\Lambda$ of $X$ which contains some singularities of the vector field is called a \emph{singular hyperbolic set} if
\begin{enumerate}
\item $\Lambda$ is \emph{partially hyperbolic} set: the vector bundle over $\Lambda$ can be decompose to the invariant \emph{stable direction} $E^s$ of dimension one and invariant \emph{central unstable direction} $E^{cu}$ of dimension two, i.e. $T_{\Lambda}M=E^s\oplus E^{cu}$
and
$$DX_t(E_x^{cu})\subset E_{X_t(x)}^{cu}\quad\text{and}\quad DX_t(E_x^{s})\subset E_{X_t(x)}^{s}$$
such that for every $t>0$ and $x\in \Lambda$, $$\|DX_t|_{E^s_x}\|<c\lambda^t\quad\text{and}\quad\|DX_t|_{E^s_x}\|.\|DX_{-t}|_{E^{cu}_{X_t(x)}}\|< c\lambda^t$$
where $c>0$ and $\lambda\in(0,1)$.
\item $\Lambda$ is volume expanding in the central unstable direction $E^{cu}$:
$$\left|det(DX_t|_{E^{cu}_x})\right|\geq ce^{-\lambda t}\quad\text{for every } t>0.$$
\item All the singularities of $\Lambda$ are hyperbolic.
\end{enumerate}

\end{defin}

\begin{defin}
A transitive compact set $\Lambda$ is called an attractor if $$\Lambda=\bigcap_{t\geq 0}X_t(U),$$
for some open set $U$ such that $\overline{X_t(U)}\subset U$, for any $t>0$. The set $U$ is called the trapping neighborhood of $\Lambda$. The set $\Lambda$ is a repeller if it is an attractor for $-X$.
\end{defin}

\begin{thm}
\cite{morales-pacifico-pujals04} Any C$^1$-robustly transitive set with finite number of singularities on a closed 3-manifold is singular hyperbolic. Moreover,
\begin{enumerate}
\item It is either a proper attractor or a proper repeller.
\item Eigenvalues at all singularities satisfy the same inequalities as the in the Lorenz geometrical model.
\end{enumerate}
\end{thm}

\begin{defin}
An invariant probability measure $\mu$ is a physical probability measure for the flow if the basin of $\mu$:
$$B(\mu)=\left\{x:~\lim_{T\rightarrow+\infty}\frac{1}{T}\int_0^T\phi(X^t(x))~dx=\int\phi~d\mu,~for~all~continuos~\phi:M\rightarrow\mathbb R\right\}$$
has positive Lebesgue measure.
\end{defin}
\begin{thm}
\cite{araujo-pacifico-pujals-viana09} Singular hyperbolic attractors support a unique physical probability measure which is ergodic and its basin covers a full Lebesgue measure subset of topological basin of attraction.
\end{thm}
We remeber that the topological basin of attraction is the set
$$\left\{x: ~\lim_{t\rightarrow+\infty}dist(X^t(x),\Lambda)=0\right\}.$$

\subsection{Poincar\'e section}
Let $S=\{\sigma_i: 1\leq i\leq s\}$ be the set of all singularities of the singular hyperbolic attractor $\Lambda$. Every singularity $\sigma_i$ is a Lorenz-like singularity, that is, the derivative $DX(\sigma_i)$ has three eigenvalues satisfying 
$$
\lambda_{i,1}>0>\lambda_{i,2}>\lambda_{i,3} \quad\text{and}\quad \lambda_{i,1}+\lambda_{i,2}>0.
$$
Around each singularity $\sigma_i$ we take two input cross-sections $\Sigma_i^{\pm,0}$ at regular points of different components of $W_{loc}^s(\sigma_i)\setminus W_{loc}^{ss}(\sigma_i)$, where $W_{loc}^{ss}(\sigma_i)$ is the stable manifold associated to the strongest contracting eigenvalue, and two output cross-section $\Sigma_i^{\pm,1}$ at regular points of different components of $W_{loc}^u(\sigma_i)\setminus \{\sigma_i\}$. By taking the linear stable and unstable eigenspaces and then local stable and unstable manifolds as coordinates, it can be seen that the local stable and unstable manifolds are contained in  $x_2x_3$-plane 	and $x_1$ axis respectively, see Figure \ref{flowbox} (for more details see \cite[Section 3.2]{wiggins88}). By descaling the coordinate, we may assume that the local stable and unstable manifolds are unit square and unit interval respectively. The cross-sections can be chosen in a way that all the trajectories started at each point on the input cross-section $\Sigma_i^{\pm,0}$ for the first time will reach the output cross-section $\Sigma_i^{\pm,1}$ except for the points on $\Gamma_i=\Sigma_i^{\pm,0}\cap W_{loc}^s(\sigma_i)$ which end at the singularity. The time is needed for the point $(x_1,x_2,x_3)$ on the input cross-section to reach the output cross-section is given by $$t_{\sigma_i}=-\frac{\log |x_1|}{\lambda_{i,1}}.$$

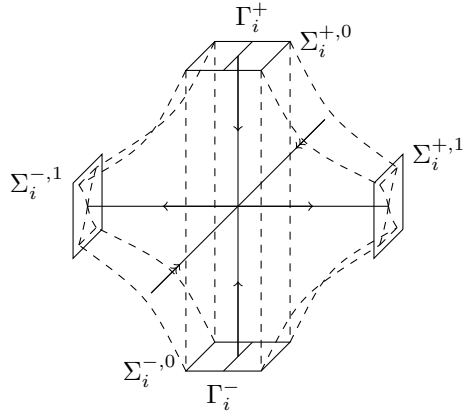
\begin{figure}[ht]
\begin{center}
\begin{tikzpicture}
\draw (-2,0,0)--(2,0,0);
\draw[->] (0,0,0)--(1,0,0);
\draw[->] (0,0,0)--(-1,0,0);
\draw (0,-2,0)--(0,2,0);
\draw[->] (0,-2,0)--(0,-1,0);
\draw[->] (0,2,0)--(0,1,0);
\draw (0,0,-3)--(0,0,3);
\draw[->>] (0,0,-3)--(0,0,-2);
\draw[->>] (0,0,3)--(0,0,2);
\draw (-.5,2,-.5)--(.5,2,-.5)--(.5,2,.5)--(-.5,2,.5)--(-.5,2,-.5);
\node[right] at (.5,2,-.5) {$\Sigma_i^{+,0}$};
\draw (-.5,-2,-.5)--(.5,-2,-.5)--(.5,-2,.5)--(-.5,-2,.5)--(-.5,-2,-.5);
\node[left] at (-.5,-2,.5) {$\Sigma_i^{-,0}$};
\draw (2,-.5,-.5)--(2,.5,-.5)--(2,.5,.5)--(2,-.5,.5)--(2,-.5,-.5);
\node[right] at (2,.5,-.5) {$\Sigma_i^{+,1}$};
\draw (-2,-.5,-.5)--(-2,.5,-.5)--(-2,.5,.5)--(-2,-.5,.5)--(-2,-.5,-.5);
\node[left] at (-2,.5,.5) {$\Sigma_i^{-,1}$};
\draw (0,2,-.5)--(0,2,.5);
\node[above] at (0,2,-.5) {$\Gamma_i^+$};
\draw (0,-2,-.5)--(0,-2,.5);
\node[below] at (0,-2,.5) {$\Gamma_i^-$};
\draw[dashed] (-.5,2,-.5)--(-.5,-2,-.5);
\draw[dashed] (.5,2,-.5)--(.5,-2,-.5);
\draw[dashed] (.5,2,.5)--(.5,-2,.5);
\draw[dashed] (-.5,2,.5)--(-.5,-2,.5);
\draw[dashed] (.5,2,.5) .. controls (1,1,.5) .. (2,.4,.3);
\draw[dashed] (.5,2,-.5) .. controls (1,1,-.5) .. (2,.4,-.3);
\draw[dashed] (2,.4,.3) -- (2,.4,-.3);
\draw[dashed] (2,.4,.3) ..controls (2,.2,.1) .. (2,0,0);
\draw[dashed] (2,.4,-.3) ..controls (2,.2,-.1) .. (2,0,0);
\draw[dashed] (.5,-2,.5) .. controls (1,-1,.5) .. (2,-.4,.3);
\draw[dashed] (.5,-2,-.5) .. controls (1,-1,-.5) .. (2,-.4,-.3);
\draw[dashed] (2,-.4,.3) -- (2,-.4,-.3);
\draw[dashed] (2,-.4,.3) ..controls (2,-.2,.1) .. (2,0,0);
\draw[dashed] (2,-.4,-.3) ..controls (2,-.2,-.1) .. (2,0,0);
\draw[dashed] (-.5,2,.5) .. controls (-1,1,.5) .. (-2,.4,.3);
\draw[dashed] (-.5,2,-.5) .. controls (-1,1,-.5) .. (-2,.4,-.3);
\draw[dashed] (-2,.4,.3) -- (-2,.4,-.3);
\draw[dashed] (-2,.4,.3) ..controls (-2,.2,.1) .. (-2,0,0);
\draw[dashed] (-2,.4,-.3) ..controls (-2,.2,-.1) .. (-2,0,0);
\draw[dashed] (-.5,-2,-.5) .. controls (-1,-1,-.5) .. (-2,-.4,-.3);
\draw[dashed] (-.5,-2,.5) .. controls (-1,-1,.5) .. (-2,-.4,.3);
\draw[dashed] (-2,-.4,.3) -- (-2,-.4,-.3);
\draw[dashed] (-2,-.4,.3) ..controls (-2,-.2,.1) .. (-2,0,0);
\draw[dashed] (-2,-.4,-.3) ..controls (-2,-.2,-.1) .. (-2,0,0);

\end{tikzpicture} 
\end{center}
\caption{flow box containing a singularity}
\label{flowbox}
\end{figure}

Let $\Sigma_x$ be a cross-section at a nonsingular point $x$ of the attractor.  We chose a flow box near each cross-section $\Sigma_x$, i.e. $U_x=\{X^t(\Sigma_x): t\in[-\epsilon,\epsilon]\}$, for some $\epsilon>0$. Around each singularity $\sigma_i$ we can take the flow box $$U_{\sigma_i}=\{X^t(x): t\in [-\epsilon,t_{\sigma_i}(x)+\epsilon]\text{ and } x\in\Sigma_i^{\pm,0}\setminus\Gamma_i\}\cup\Sigma_i^{+,0}\times[-1,1].$$ The attractor has the open cover $(\bigcup\limits_{x\in \Lambda\setminus S}U^\mathrm{o}_x)\cup(\bigcup\limits_{\sigma_i\in S}U^\mathrm{o}_{\sigma_i})$. Since the attractor is compact, it has a finite cover $\bigcup\limits_{j=1}^l U_{x_j}$, where $l\geq s$ and $x_j=\sigma_j$ for $j\in\{1,\ldots,s\}. $
The \emph{global cross-section} or \emph{Poincar\'e section} is the family of the cross-sections $$\Xi=\left\{\Sigma_i^{\pm,0},\Sigma_i^{\pm,1}, \Sigma_{x_j}:i=1,\ldots,s~\text{and}~j=s+1,\ldots,l\right\}.$$

\subsection{Stable foliation}
Note that each cross-section $\Sigma$ in the Poincar\'e section $\Xi$ can be foliated by stable leaves $W^s(x,\Sigma)$ for $x\in\Sigma$ where the leaf $W^s(x,\Sigma)$ is the connected component of $\bigcup_{t\in\mathbb{R}}X^t(W^{s}(x))\cap \Sigma$ that contains the point $x$ for which $W^s(x)$ is the stable manifold of $x$.

\subsection{Return time}
The return time $\tau(x)$ for $x$ on Poincar\'e section is defined to be the first time that $X^t(x)$ hits the global cross-section when $t>t_0$, i.e.
$$\tau(x)=t_0+T(X^{t_0}(x))$$
where $T(x)=\inf\{t>0:X^t(x)\in\Xi\}$ is the first hitting time. The fixed time $t_0$ is big enough in order that the Poincar\'e map become uniformly hyperbolic.

\subsection{Integrability of the return time}
Note that the return time is not defined for some finite leaves on Poincar\'e section which has Lebesgue measure zero. The return time is bounded by $$t_0+\sum\limits_{i=1}^{s}t_{\sigma_i}\mathbb{I}_{X^{-t_0}(U_{\sigma_i})}+\sum\limits_{i=s+1}^{l}2\epsilon\mathbb{I}_{X^{-t_0}(U_{x_i})},$$ where $\mathbb{I}$ is the indicator function. Since the times $t_{\sigma_i}$'s are integrable, the return time is integrable with respect to the Lebesgue measure.

\subsection{Return map}
Now the Poincar\'e return map $R:\Xi \to \Xi$ is defined as
$$R(x)=X^{\tau(x)}(x).$$
In order that the Poincar\'e map leaves invariant the stable foliation on the Poincar\'e section (image of each stable leaf is entirely in another stable leaf), each cross-section at a nonsingular point of the attractor can be chosen to be $\delta$-adapted, that is, the distance of its intersection with the attractor stay out side of the $\delta$ neighborhood of the $cu$-boundary of the cross-section. 
\begin{center}
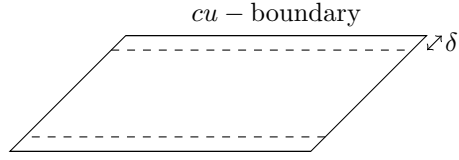
\begin{figure}[h]
\begin{tikzpicture}
\draw (-2,0,-2)--(2,0,-2)--(2,0,2)--(-2,0,2)--(-2,0,-2);
\draw (-2,0,-2)--(2,0,-2);
\node [above] at (0,0,-2) {$cu-\text{boundary}$};
\draw (2,0,2)--(-2,0,2);
\draw [dashed](-2,0,-1.5)--(2,0,-1.5);
\draw [<->](2.2,0,-2)--(2.2,0,-1.5);
\node [right] at (2.2,0,-1.8) {$\delta$};
\draw [dashed](2,0,1.5)--(-2,0,1.5);
\end{tikzpicture}
\caption{$\delta$-adapted cross-section}
\end{figure}
\end{center}
Then a point $x\in \Xi\setminus \partial^s\Xi$ is a discontinuity point of the Poincar\'e map $R$ if $R(x)\in\partial\Xi$ or $X_t(x)\in W^s(\sigma)$ for some $t\in [0,t_0]$ and a singular point $\sigma$. The set of the discontinuities is a finite numbers of the stable leaves on the global cross-section and therefore the domain of smoothness of $R$ is some open strips.

\subsection{Quotient function}
From now on, it is assumed that the flow is of class $C^2$. Under this condition, it can be shown the stable leaves $W^s(x)$ are $C^2$ embedded disks which define a $C^{1+\alpha}$ foliation (see \cite{araujo-melbourne17}). We choose a $C^2$ $cu$-leaf transversal to the stable foliation on each cross-section of Poincar\'e section. By collapsing the stable leave on $cu$-leaves, a $C^{1+\alpha}$ function is obtained (excluding the discontinuity leaves of the return map) whose domain is $C^2$ diffeomorphic to $I\subset\mathbb{R}$ containing a finite union of open intervals. Therefore the projection function $p:\Sigma\to I$ is a $C^{1+\alpha}$ and so is the function 
\begin{align*}
f&:I\to I\\
f(x)&=p(R(p^{-1}(x))).
\end{align*}
Note that the transformation $f$ is piecewise expanding and $\frac{1}{f'}$ is $\alpha$ continuous. As a result it can be proved that the function $f$ has a finite number of absolutely continuous invariant measures whose
basins cover Lebesgue almost all points of $I$ \cite{keller85}. The function $f$ is transitive since the attractor is transitive. Therefore $f$  has a unique absolutely continuous invariant probability measure $\nu$.
We remember that, since the map $f$ is $C^{1+\alpha}$ and each H\"older continuous function with H\"older exponent $\alpha\in(0,1]$ is bounded $\frac{1}{\alpha}$-variation, the map $\frac{1}{f'}$ is bounded $\frac{1}{\alpha}$-variation. 
A function $f$ is bounded $\frac{1}{\alpha}$-variation if
$$V_{\frac{1}{\alpha}}(f)=\sup\limits_{0\leq x_0<\ldots <x_n\leq 1}\left(\sum_{i=1}^{n}|f(x_i)-f(x_{i-1})|^{\frac{1}{\alpha}}\right)^{\alpha}<+\infty~.$$

\subsection{Liftting SRB measures}
As $f\circ p=p\circ R$ the measurable and topological structures of $f$ can be induced to $R$. In particular, there is an invariant absolutely continuous probability measure on the Poincar\'e section defined by $$\int g~d\gamma=\lim\limits_{n\to\infty}\int\inf\limits_{x\in\zeta}g( R^n(x))~d\nu=\lim\limits_{n\to\infty}\int\sup\limits_{x\in\zeta}g(R^n(x))~d\nu,$$ 
for every continuous and bounded function $g$ on the Poincar\'e section.

By saturating the measure $\hat\mu$ along the flow, an invariant probability measure $\mu$ can be construct supported on the attractor defined by 
$$\int h~d\mu=\frac{1}{\gamma(\tau)}\int\int_0^{\tau(x)}h(X^t(x))~dt~d\gamma(x),$$
for every continuous and bounded function $h$ on a neighborhood of the attractor. The measure $\mu$ is a unique ergodic and SRB measure supported on the attractor.

\subsection{Perturbed singular hyperbolic attractor}
Let $X^t$ be the flow of a $C^2$ vector field $X$ supported a singular hyperbolic attractor $\Lambda$ as a maximal invariant set in a open set $U$. Let $X_\epsilon$ be such a small $C^2$ perturbation of $X$ for which the flow $X^t_\epsilon$ supports a singular hyperbolic attractor $X_\epsilon$ with same number of singularity. 
Since the global Poincar\'e section $\Xi$ is transversal to flow direction and is $\delta$-adapted, choose $\Xi$ as a global Poincar\'e section for the flow of the vector filed $X_\epsilon$:
the Poincar\'e section is foliated by stable leaves and collapsing along them an expanding $C^{1+\alpha}$ function $f_\epsilon:I\to I$ is obtained with a unique absolutely invariant probability measure $\nu_\epsilon$.
Then any map $f_\epsilon$ is transitive and piecewise $C^{1+\alpha}$ expanding with the same number of discontinuity points of $f$. Moreover, there are uniform constants $\beta>1$ and $c>0$ such that

$$(f^n_\epsilon)'(x)\geq\beta^n$$
and

$$V_{\frac{1}{\alpha}}(\frac{1}{f'_\epsilon})\leq c.$$

By constructing an invariant probability measure and saturating it along the flow, an invariant probability measure $\mu_\epsilon$ supported on the attractor $\Lambda_\epsilon$ will be obtained. Our objective here is to compare $\mu$ with $\mu_\epsilon$ and to show that they are close in weak$^*$ topology. Our method is to show that the measures $\nu$ and $\nu_\epsilon$ of the reduced one-dimensional dynamics are close The closeness, by \cite{alves-soufi14}, will be preserved on the Poincar\'e section and the relative attractor.. 

\section{Statistical stability for one-dimensional map}
The transform operator describes well the measurable properties of a dynamics: the invariant measure of the function $f$ is the fixed point of the transform operator acting on a Banach space defined by $$T(g(x))=\sum\limits_{f(y)=x}\frac{g(y)}{|\text{Det } D_yf|}.$$ The transform operator acting on $L^1$ space may not have spectral gap. We should look for an appropriate Banach  subspace of $L^1$ with a norm stronger than $L^1$-norm such that the action of the transform operator on this Banach space has spectral gap. Note that when an operator is quasi-compact, it has spectral gap. A sufficient condition for quasi-compactness is given by Losato-Yorke inequality: 
there are $k\in\mathbb N$ and uniform constants $0<r<1$ and $R>0$ such that
$$||T^kg||\leq r^k||g||+R|||g|||$$

\subsection{Bounded variation space}
For a function $g:I\to\mathbb{R}$ and $\epsilon>0$ the oscillation of $g$ is defined as following
$$\osc(g,\epsilon,x)=\esssup\limits_{y_1,y_2\in B_{\epsilon}(x)}|g(y_1)-g(y_2)|$$
where $B_{\epsilon}(x)$ is the open ball with radius $\epsilon$ centered at $x$. The oscillation is semi-continuous and therefore measurable. We then define $$\osc(g,\epsilon)=\|\osc(g,\epsilon,x)\|_1.$$ 
Fix $\epsilon_0>0$ and, for $n\in\mathbb N$, set  $\mathcal R_{n}=\{g:I\to\mathbb{R}~|~\osc(g,\epsilon)\leq n\epsilon,~\epsilon\in(0,\epsilon_0]\}$ and $\mathcal R=\bigcup\limits_{n\in\mathbb{N}}\mathcal R_{n}$. 
It can be shown that $\mathcal R$ is dense in $(L_1,\|.\|_1)$. $\mathcal B \mathcal V$ is the space of equivalence classes of functions in $\mathcal R$. 
Define 
$$\var_\alpha(g)=\sup\limits_{0<\epsilon\leq\epsilon_0}\frac{\osc(g,\epsilon)}{\epsilon^\alpha}$$ 
and
$$||g||_{\alpha,1}=\var_\alpha(g)+\|g\|_1.$$ 
The space $(\mathcal B \mathcal V,||.||_{\alpha,1})$ is a Banach space with the compact embedding into $L^1$.

we must show that all constants are uniform, that is, they do not depend on $f$.
\begin{lem}
For any $g\in \mathcal B \mathcal V$ and any $J\subset I$ with $|J|\geq\epsilon_0$, we have
$$\osc(g.1_J,\epsilon)\leq 2~\int_J\osc(g|_J,\epsilon,x)~dx+\frac{4\epsilon}{|J|}~\int_J|g(x)|~dx.$$
\end{lem}
\begin{proof}
Let $I=[a_1,a_2]$. We observe the following
$$\osc(g.1_J,\epsilon,x)=H_1+\osc(g|_J,\epsilon,x).1_{J_0}(x)+H_2$$
where $H_i=\text{max} \left\{\esssup\limits_{y\in B_\epsilon(x)\cap J_i}|g(y)|.1_{J_i}(x),\osc(g|_J,\epsilon,x).1_{J_i}(x)\right\}$, $J_1=B_\epsilon(a_1)$, $J_2=B_\epsilon(a_2)$ and $J_0=I\setminus J_1\cup J_2$.\\
There are two cases:\\
i) If $H_i=\osc(g|_J,\epsilon,x).1_{J_i}(x)$, for $i=1,2$,  then it is done.\\\\
ii) If for some $i$, $H_i=\esssup_{y\in B_\epsilon(x)\cap J_i}|g(y)|.1_{J_i}(x)$.\\
Then there exists some $k\in\{1,...,n-1\}$ such that
$$\essinf_{y\in B_\epsilon (a_1+2k\epsilon+z)}g(y)\leq g_0\leq\esssup_{y\in B_\epsilon(a_1+2k\epsilon+z)}g(y)$$
for $g_0=\frac{1}{m(J)}\int_Jg(x)dx$.
Since, for any $y\in B_\epsilon(x)\cap J$,
$$|g(y) - g_0|\leq |g(y)-g(a_1+z)|+|g(a_1+z)-g(a_1+2\epsilon+z)|+...$$
$$+|g(a_1+2(k-1)\epsilon+z)-g(a_1+2k\epsilon+z)|+|g(a_1+2k\epsilon+z)-g_0|$$
$$\leq\sum_{i=1}^{k}\osc(g|_J,\epsilon,a_1+2i\epsilon+z)$$

then
$$\osc(g.1_J,\epsilon)\leq \int\osc(g|_J,\epsilon,x).1_{J_0}(x)+4\epsilon|g_0|+\int_{-\epsilon}^{\epsilon}\sum_{i=1}^{k}\osc(g|_J,\epsilon,a_1+2i\epsilon+z)dx$$
$$\leq 2\int\osc(g|_J,\epsilon,x)~dx+\frac{4\epsilon}{|J|}\int|g(x)|dx$$
\end{proof}

The next proposition, as a direct consequence of the last lemma, approximates the variation of a transform operator by variation of the base transformation.
\begin{prop}
There are uniform positive constants $c_1$ and $c_2$ such that for any $g\in \mathcal B \mathcal V$,
$$\var_\alpha(T g)\leq c_1~\var_\alpha(g)+c_2~||g||_1.$$
\end{prop}
\begin{proof}
We have 2 steps.\\
Step 1. We show that
$$\int_{f(J)}\osc(\frac{g\circ f^{-1}}{f^\prime\circ f^{-1}}|_{f(J)},\epsilon,y)dy\leq\int_J\osc(g|_J,\beta^{-1}\epsilon,x)dx+5||g|_J||_\infty\int_J\osc(\frac{1}{f^\prime}f^{-1}|_{f(J)},\epsilon,y)dy$$
Since
$$\osc(\frac{g\circ f^{-1}}{f^\prime\circ f^{-1}}|_{f(J)},\epsilon,y)\leq||g(f^{-1}(y))||~\osc(\frac{ 1}{f^\prime\circ f^{-1}}|_{f(J)},\epsilon,y)$$
$$+\frac{1}{f^\prime(f^{-1}(y))}~\osc(g\circ f^{-1}|_{f(J)},\epsilon,y)$$
$$+2~\osc(\frac{1}{f^\prime\circ f^{-1}}|_{f(J)},\epsilon,y)~\osc(g\circ f^{-1}|_{f(J)},\epsilon,y)$$
Over the second term, by change of variable, we have 
$$\int_{f(J)}\osc(g\circ f^{-1}|_{f(J)},\epsilon,y)~f^{-1}(y)~dy$$
$$\leq\int_J\osc(g|_J,\beta^{-1}\epsilon,x)~dx.$$
We conclude observing that
$$\int_{f(J)}\osc(\frac{ f^{-1}}{f^\prime\circ f^{-1}}|_{f(J)},\epsilon,y)~[~||g(f^{-1}(y))||+\osc(g\circ f^{-1}|_{f(J)},\epsilon,y)~]~dy$$
$$\leq 5~||g|_J||_\infty\int_J\osc(\frac{1}{f^\prime}f^{-1}|_{f(J)},\epsilon,y)dy$$
Step 2. By the last lemma, we have
$$\osc(Tg,\epsilon)\leq\sum_{j=1}^N \osc(\frac{g\circ f^{-1}}{f^\prime\circ f^{-1}}1_{f(I_j)},\epsilon)$$
$$\leq\sum_{j=1}^N2\int_{f(I_j)}\osc(\frac{g\circ f^{-1}}{f^\prime\circ f^{-1}}1_{f(I_j)},\epsilon)~dy+\frac{4\epsilon}{|I_j|}\int_{I_j}|g(x)|dx$$
$$\leq 2\int_{I}\osc(g,\beta^{-1}\epsilon,x)dx+5||g||_\infty\int_{I}\osc(\frac{1}{f^\prime}f^{-1},\epsilon,y)dy+\frac{4\epsilon}{|I_m|}\int_{I}|g(x)|dx,$$
where $|I_m|\leq|I_j|$, for all $j$. Then
$$\frac{\osc(Tg,\epsilon)}{\epsilon^{\alpha}}\leq 2\beta^{-\alpha}\var_\alpha(g)+5~||g||_\infty\var_\alpha(\frac{1}{f^\prime})+\frac{4}{|I_m|}\epsilon_0^{1-\alpha}\int_{I}|g(x)|dx$$
$$\leq 2\beta^{-\alpha}\var_\alpha(g)+\frac{5c}{\epsilon_0}\var_\alpha(g)+5c\int|g(x)|dx+\frac{4}{|I_m|}\epsilon_0^{1-\alpha}||g||_1,$$
since $||g||_\infty\leq\frac{1}{\epsilon_0}\var_\alpha(g)+\int|g(x)|dx$.\\
Take $c_1=2\beta^{-\alpha}+\frac{5c}{\epsilon_0}$ and $c_2=5c+\frac{4}{|I_m|}\epsilon_0^{1-\alpha}$.
\end{proof}
\subsection{The Lasota-Yorke inequality}
The next proposition, as a direct consequence of the last proposition, describes a rather exponential behavior of a transform operator.
\begin{prop}
There exist uniform positive constants $c_3,c_4$ and $\lambda\in(0,1)$ such that for any $g\in \mathcal B \mathcal V$ 
$$||T^ng||_{\alpha,1}\leq c_3\lambda^n||g||_{\alpha,1}+c_4||g||_1,$$
for all $n\in\mathbb N$.
\end{prop}
\begin{proof}
We have 3 steps\\
Step 1. We prove that there are $k\in\mathbb N$, a uniform constant $c_5>0$ and $\hat\lambda\in[0,1]$ such that for any $g\in \mathcal B \mathcal V$,
$$||T^kg||_{\alpha,1}\leq \hat\lambda ||g||_{\alpha,1}+c_5||g||_1.$$
There exists $k\in\mathbb N$ such that $\hat\lambda=2\beta^{-k\alpha}+\frac{5kc}{\epsilon_0\beta^{k-1}}<1$ then, by last proposition,
$$||T^kg||_{\alpha,1}=\var_\alpha(T^kg)+||T^kg||_1\leq (2\beta^{-k\alpha}+\frac{5kc}{\epsilon_0\beta^{k-1}})\var_\alpha(g)+(\frac{5kc}{\beta^{k-1}}+\frac{4\epsilon_0^{1-\alpha}}{|I_{m,k}|})||g||_1,$$
since $T^k_f=T_{f^k}$ and $var(\frac{1}{(f^k)'})\leq\frac{k}{\beta^{k-1}}\var(\frac{1}{f'})$.\\
Take $c_5=\frac{5kc}{\beta^{k-1}}+\frac{4\epsilon_0^{1-\alpha}}{|I_{m,k}|}$.
\\\\
Step 2. Let $1\leq r\leq k$, then, by the last proposition, we have
$$\var_\alpha(T^rg)\leq c_1 \var_\alpha(T^{r-1}g)+c_2||T^{r-1}g||_1\leq ...$$
$$...\leq c_1^r \var_\alpha(g)+(c_2\sum_{j=0}^{r}c_1^j) ||g||_1$$
$$\leq \hat\lambda^{r-k}c_1^r \var_\alpha(g)+(c_2\sum_{j=0}^{r}c_1^j) ||g||_1$$
Since $\|T^rg\|_1\leq \|g\|_1$, we see that 
$$||T^rg||_{\alpha,1}\leq (\hat\lambda^{-k}c_1^r)\hat\lambda^{r}||g||_{\alpha,1}+ (c_2\sum_{j=0}^{r}c_1^j) ||g||_1.$$
Step 3. If $n>k$ then one may take $n=mk+r$ for some $m\in\mathbb N$ and $0<r< k$. As we have
$$||T^n g||_{1,\alpha}=||(T^k)^m\circ T^rg||_{1,\alpha}\leq\hat\lambda||(T^k)^{m-1}\circ T^rg||_{1,\alpha}+c_5||(T^k)^{m-1}\circ T^rg||_1\leq...$$
$$...\leq\hat\lambda^m||T^rg||_{1,\alpha}+c_5\sum_{j=0}^m\hat\lambda^j||T^rg||_1$$
$$\leq(\hat\lambda^{r-k}c_1^r)\hat\lambda^{m} ||g||_{1,\alpha}+ [\hat\lambda^m(c_2\sum_{j=0}^{r}c_1^j)+c_5\sum_{j=0}^m\hat\lambda^j]~||g||_1$$
$$\leq(\hat\lambda^{r-k-\frac{r}{k}}c_1^r)(\hat\lambda^{\frac{1}{k}})^n||g||_{1,\alpha}+ [\hat\lambda^m(c_2\sum_{j=0}^{r}c_1^j)+c_5\sum_{j=0}^m\hat\lambda^j]~||g||_1.$$
Take $\lambda=\hat\lambda^\frac{1}{k}$, $c_3=\max\{\hat\lambda^{r-k-\frac{r}{k}}c_1^r,\hat\lambda^{-k}c_1^r\}=\hat\lambda^{-k}c_1^k$ and $c_4=\max\{c_2\sum_{j=0}^{r}c_1^j,\hat\lambda^m(c_2\sum_{j=0}^{r}c_1^j)+c_5\sum_{j=0}^m\hat\lambda^j\}\leq c_2\frac{c_1^{k+1}-1}{c_1-1}+c_5\frac{1}{1-\hat\lambda}$.

\end{proof}
\subsection{Continuity of the transform operators}
we define
$$|||T_{\epsilon}|||:=\sup_{||f||_{1,\alpha}\leq 1}||T_{\epsilon} f||_1$$
\begin{lem}
If $g\in \mathcal{BV}_\alpha$ and if $\mathcal{I}=\{I_1, I_2,\ldots, I_n\}$ is a partition of subintervals of the interval $I=[0,1]$, then
$$||g-g_n||_1\leq\epsilon^\alpha\var_\alpha(g),$$
where $g_n=\sum\limits_{j=1}^{n}\left(\frac{1}{|I_j|}\int_{I_j}g(x)dx\right)~~\mathbb{I}_{I_j}$ and $\epsilon=\max\limits_{j}|I_j|\leq\epsilon_0$. 

\end{lem}
\begin{proof}
	For every $x$ and $y$ in subinterval $I_j$, we have $$|g(x)-g(y)|\leq \osc(g,x,|I_j|)\leq \osc(g,x,\epsilon).$$
	By integrating the above inequality with respect to $y$, it holds that
	
	$$\left|g(x)-\frac{1}{|I_j|}\int_{I_j}g(y)~dy\right|\leq\osc(g,x,\epsilon).$$
	From this we conclude that
	\begin{align*}
	\|g-g_n\|_1&\leq \sum\limits_{j=1}^n\int_{I_j}\left|g(x)-\frac{1}{|I_j|}\int_{I_j}g(y)~dy\right|~dx\\&\leq \sum\limits_{j=1}^n\int_{I_j}\osc(g,x,\epsilon)~dx=\int\osc(g,x,\epsilon)~dx\\&\leq \epsilon^\alpha var_\alpha(g).
	\end{align*}
	
\end{proof}

\begin{lem}
Let $g\in \mathcal{BV}_\alpha$ and $\phi\in L_{\infty}$. There exists a constant $C$ such that for any $\epsilon\in(0,\epsilon_0]$, the following inequality holds,
$$\left|\int g\phi~dx\right|\leq \epsilon^\alpha \|g\|_{\alpha,1}\|\phi\|_\infty+\frac{C}{\epsilon^{1-\alpha}}\|g\|_{\alpha,1}\|\Phi\|_\infty,$$
where $\Phi(z)=\int_{x\leq z}\phi(x)~dx$.
\end{lem}
\begin{proof}
Take $I_1=(a_1,b_1),\dots, I_N=(a_n,b_n)$ be a partition of $I$ into sub intervals such that $|I_j|\leq\epsilon_0$ and $\frac{|I_j|}{|I_k|}\leq c$, for some $c>1$, and any $j,k$.  
Let $\bar g_j=\frac{1}{|I_j|}\int_{I_j}g(x)dx$, $\epsilon=\max\limits_{j}|I_j|$ and $\delta=\min\limits_{j}|I_j|$. The triangle and H\"older inequalities imply that
\begin{align*}
\left|\int g\phi~dx\right|\leq\left|\int g\phi~dx-\int g_n\phi~dx\right|+\left|\int g_n\phi~dx\right|\leq \|g-g_n\|_1\|\phi\|_{\infty}+\left|\int g_n\phi~dx\right|
\end{align*}
On the other hand
\begin{align*}
\left|\int g_n\phi(x)~dx\right|&=\left|\sum_{j=1}^{n}\int_{I_j} g_n\phi(x)~dx\right|=\left|\sum_{j=1}^{n}\bar g_j\int_{I_j}\phi(x)~dx\right|=\left|\sum_{j=1}^{n}\bar g_j\left(\Phi(b_j)-\Phi(a_j)\right)\right|\\
&\leq |\Phi(0)~\bar g_1|+|\Phi(1)~\bar g_n|+\sum_{j=2}^{n}|\bar g_{j}-\bar g_{j-1}|~||\Phi||_{\infty}\quad (\Phi(0)=0)\\
&\leq\|\Phi\|_{\infty}~\|g\|_\infty+\|\Phi\|_{\infty}\frac{4c2^\alpha}{\epsilon^{1-\alpha}}\var_{\alpha}(g)\\
&\leq \|\Phi\|_{\infty}\left(\frac{1}{\epsilon_0}+\frac{4c2^\alpha}{\epsilon^{1-\alpha}}\right)\|g\|_{\alpha,1}\quad \left(\|g\|_\infty\leq\frac{1}{\epsilon_0}\var_\alpha(g)+\|g\|_1\leq \frac{1}{\epsilon_0}\|g\|_{\alpha,1}\right)\\
&\leq \frac{C}{\epsilon^{1-\alpha}}\|\Phi\|_{\infty}\|g\|_{\alpha,1}
\end{align*}
Indeed, for any $j$,
\begin{align*}
|\bar g_{j}-\bar g_{j-1}|&\leq|\bar g_{j}-\bar g_{j,j-1}|+|\bar g_{j,j-1}-\bar g_{j-1}|\quad \left(\bar g_{j,j-1}=\frac{1}{|I_j|+|I_{j-1}|}\int_{I_j\cup I_{j-1}}g(x)~dx\right)\\
&=\frac{1}{|I_j|}\left|\int_{I_j}g(x)-\bar g_{j,j-1}~dx\right|+\frac{1}{|I_{j-1}|}\left|\int_{I_{j-1}}g(x)-\bar g_{j,j-1}~dx\right|\\
&\leq \frac{1}{|I_j|}\int_{I_j\cup I_{j-1}}|g(x)-\bar g_{j,j-1}|~dx+\frac{1}{|I_{j-1}|}\int_{I_j\cup I_{j-1}}|g(x)-\bar g_{j,j-1}|~dx\\
&\leq \frac{2}{\delta}\int_{I_j\cup I_{j-1}}\osc(g,x,2\epsilon)dx.
\end{align*}
Then
\begin{align*}
\sum_{j=2}^n |\bar g_{j}-\bar g_{j-1}|&\leq \frac{4}{\delta}\int\osc(g,x,2\epsilon)dx\leq \frac{4}{\delta}(2\epsilon)^{\alpha}\var_{\alpha}(g)\leq \frac{4c2^\alpha}{\epsilon^{1-\alpha}}\var_{\alpha}(g)
\end{align*}

\end{proof}
\begin{defin}
	The distance between $f_1$ and $f_2$, denoted by $d(f_1,f_2)$, is defined to be the infimum of $\epsilon > 0$ for which there exist a subset $A\subset I$ with $|A| > 1 - \epsilon$ and a diffeomorphism $\sigma :I\to I$ such  that $f_2|_A = f_1\circ \sigma|_A$, and
    for all $x\in A,~~ |\sigma(x) - x|< \epsilon$ and $|\frac{1}{\sigma'(x)} - I| < \epsilon$.
\end{defin}
\begin{prop}
$$\lim_{\epsilon\rightarrow 0}|||T_\epsilon-T_0|||=0$$
\end{prop}
\begin{proof}
For any $g\in \mathcal B\mathcal V$ and let $H=\frac{|T_\epsilon g-T_0g|}{T_\epsilon g-T_0g}$ we have
\begin{align*}
\int|T_\epsilon g-T_0g|~dx&=\int H~(T_\epsilon g-T_0g)~dx=\int g~(H\circ f_\epsilon ~-~H\circ f_0)~dx\\
&\leq \epsilon^\alpha\|g\|_{\alpha,1}\|H\circ f_\epsilon ~-~H\circ f_0\|_{\infty}+\frac{C}{\epsilon^{1-\alpha}}\|g\|_{\alpha,1}\sup_z\left|\int_0^z(H\circ f_\epsilon~-~H\circ f_0)~dx\right|\\
&\leq 2\epsilon^\alpha\|g\|_{\alpha,1}+\frac{C}{\epsilon^{1-\alpha}}\|g\|_{\alpha,1}\sup_z\left(2\epsilon+\left|\int_{A\cap [0,z]}(H\circ f_\epsilon\circ\sigma~-~H\circ f_0)~dx\right|\right)\\
&\leq 2\epsilon^\alpha\|g\|_{\alpha,1}+\frac{C}{\epsilon^{1-\alpha}}\|g\|_{\alpha,1}(2\epsilon+\epsilon+\epsilon)\leq (2+4C)\epsilon^\alpha\|g\|_{\alpha,1}.
\end{align*}

\end{proof}

\begin{rem}
Note that $d(f_\epsilon,f_0)=o(\epsilon)$.
\end{rem}
\subsection{Spectral Stability}
To conclude the proof of statistical stability for one dimensional dynamics $f$ of a singular hyperbolic flow, we need to show that the densities of the unique probability measures are close, as follows.
\begin{prop}
$$\lim_{\epsilon\rightarrow 0}||h_\epsilon-h_0||_1=0$$
where $h_\epsilon$ is the density of the invariant measure for $f_\epsilon$.
\end{prop}
\begin{proof}
Note that by \cite{keller-liverani99}, when $\lim\limits_{\epsilon\to 0}|||T_\epsilon-T_0|||=0$, it implies that $$\lim\limits_{\epsilon\to 0}\|\Pi_{\epsilon}-\Pi_{0}\|_1=0$$ where $\Pi_\epsilon$ and $\Pi_0$ are spectral projections associated with the eigenvalue 1. Since the eigenspaces associated with the eigenvalue 1 have dimension one, the proof is completed.
\end{proof}

\bibliographystyle{alpha}
\bibliography{ref}

\end{document}